\documentclass[12pt]{article}
\usepackage{amsmath}
\usepackage{graphicx,psfrag,epsf}
\usepackage{enumerate}
\usepackage{psfrag,epsf}
\usepackage{url} % not crucial - just used below for the URL
\usepackage{amsmath,amssymb,amsfonts,amsthm,mathtools}
\usepackage{bm,bbm,mathrsfs}
\usepackage{color}
\usepackage{subfigure}
\usepackage{algorithm,algpseudocode}

\usepackage{scalefnt}
\usepackage{multirow}
\usepackage{authblk} % author and affiliations
\usepackage[colorlinks=true, citecolor=blue, urlcolor=blue]{hyperref}
\usepackage{natbib}
\usepackage{dsfont}
\usepackage[title]{appendix}
\usepackage[title]{appendix}

\input cyracc.def

\usepackage[left=1in,top=1.1in,right=0.5in,bottom=1in]{geometry}

\usepackage{letltxmacro}
\makeatletter
\let\oldr@@t\r@@t
\def\r@@t#1#2{%
	\setbox0=\hbox{$\oldr@@t#1{#2\,}$}\dimen0=\ht0
	\advance\dimen0-0.2\ht0
	\setbox2=\hbox{\vrule height\ht0 depth -\dimen0}%
	{\box0\lower0.4pt\box2}}
\LetLtxMacro{\oldsqrt}{\sqrt}
\renewcommand*{\sqrt}[2][\ ]{\oldsqrt[#1]{#2}}
\makeatother
\theoremstyle{definition}
\newtheorem{remark1}{Remark}
\newtheorem{example1}{Example}
\newtheorem{theorem1}{Theorem}

\newtheorem{theorem}{Theorem}[section]

\newtheorem{corollary}[theorem]{Corollary}

\newtheorem{proposition}[theorem]{Proposition}

\numberwithin{equation}{section} % requires amsmath

\makeatletter
\def\@seccntformat#1{\@ifundefined{#1@cntformat}%
	{\csname the#1\endcsname\quad}%      default
	{\csname #1@cntformat\endcsname}%    enable individual control
}
\makeatother

\newif\ifShowComments
\ShowCommentstrue
\def\strutdepth{\dp\strutbox}
\def\druk#1{\strut\vadjust{\kern-\strutdepth
        {\vtop to \strutdepth{%
                \baselineskip\strutdepth\vss
                        \llap{\hbox{#1}\quad}\null}}}}

\title{\bf
On the distribution of a random variable involved in an independent ratio
}
\smallskip 
\author[1, 2]{Roberto Vila \thanks{rovig161@gmail.com}}
\author[2]{Narayanaswamy Balakrishnan  \thanks{bala@mcmaster.ca}}
\author[3]{Marcelo Bourguignon \thanks{m.p.bourguignon@gmail.com}}
\affil[1]{Department of Statistics, University of
	Bras\'ilia, Bras\'ilia, Brazil}
\affil[2]{
	Department of Mathematics and Statistics, McMaster University, Hamilton, Ontario, Canada}
\affil[3]{
	Department of Statistics, Federal University of Rio Grande do Norte, Natal, RN, Brazil}
\begin{document}
\maketitle
\smallskip 
\begin{abstract}
%In this paper, we obtain some characterizations of the law of a random variable $X$ that composes a type I ratio $X/(X+Y)$ provided that $Z$ and %$X/(X+Y)$ are equal in distribution, $X$ and $Y$ are independent, and that $Z$ and $Y$ have known distributions. As applications of the obtained %results, a variety of examples are presented.
In this paper, using inverse integral transforms,
we derive the exact distribution of the
random variable $X$ that is involved in
the ratio $Z \stackrel{d}{=} X/(X+Y)$ where $X$ and $Y$ are  independent random variables having the same
support, and $Z$ and $Y$ have known distributions. We introduce new distributions this way. As applications of the obtained
results, several examples are presented.
\end{abstract}

%\tableofcontents

\vspace*{0.75cm}
\section{Introduction}
\noindent

Let $X, Y$ and $Z$ be (absolutely) continuous and positive random variables
such that
\begin{align}\label{main-id}
Z\stackrel{d}{=}\dfrac{X}{X+Y},
\end{align}
being $\stackrel{d}{=}$ equality in distribution, and the ratio in \eqref{main-id} has support in the unit interval $(0, 1)$. The random variable $Z = X/(X+Y)$ is known in the literature as type I ratio \citep{Johnson95,Bekker2009}.

\newpage
The distributions of ratios of random variables are of interest in many fields \citep{Nadarajah2006}.
An important recent example of ratios of random variables
is in the case fatality rate of Covid-19 \citep{bour22},  where $X \in \mathbb{R}^+$ and
$Y \in \mathbb{R}^+$ are two random variables representing the number of confirmed
Covid-19-related deaths and Covid-19 cases with no death result, respectively. The sum $X + Y$ represents the number
of confirmed Covid-19 cases. It should be noted that stochastic representations are important since they
may justify some models arising naturally in real situations, as described above.

The distribution of $Z$ has been studied by several authors especially when $X$ and $Y$ are independent
random variables and come from the same family of distributions. To the best of our knowledge, there are no previous works when $X$ and $Y$  belong to different families.
\cite{malik:67} and \cite{ahuja:69} both 
discussed the case when
$X$ and $Y$ are independent random variables following
gamma distributions with shape parameters $\alpha > 0$ and $\beta > 0$, and same scale parameter $\theta$.
In a similar way, if $X \sim \mathrm{Gamma}(\alpha, \theta_1)$
and $Y \sim \mathrm{Gamma}(\beta, \theta_2)$ are independent gamma variables, then 
$Z$ is distributed according to a Libby-Novick distribution \citep{LN82}.
For a recent discussion of some extensions of this idea, one may refer to \cite{Jones2021}.
\cite{Lijoi05} considered the ratio using inverse Gaussian random variables instead of gamma random variables, and termed it as normalized inverse Gaussian
distribution. Specifically, the
normalized inverse Gaussian distribution is obtained by the stochastic representation (\ref{main-id}) with $X$
and $Y$ being independent inverse Gaussian random variables with scale parameter 1 and
shape parameters $\alpha > 0$ and $\beta > 0$.

Recently, new families of distributions have been introduced for modeling bounded quantities.
Some of the bounded distributions in the literature are derived from standard distributions by
mathematical transformation like $Z = \exp(-W)$, $Z = W/(1-W)$ or $Z = 1/(1-W)$, where
$W \in \mathbb{R}^+$. For example, the following transformation gives rise to distributions
on the unit interval: if $Z = \textrm{e}^{-W}$, where $W \sim \textrm{Exponentiated-Exponential}(\alpha, \beta)$,
$W \sim \textrm{Gamma}(\alpha, \beta)$ and
$W \sim \textrm{Lindley}(\alpha, \beta)$, implies $Z \sim \textrm{Kumaraswamy}(\alpha, \beta)$ \citep{kuma80},
$Z \sim \textrm{Unit-Gamma}(\alpha, \beta)$ \citep{Grassia77} and
$Z \sim \textrm{Log-Lindley}(\alpha, \beta)$ \citep{deniz14}, respectively, where $\alpha, \beta > 0$.
As far as we know, the Kumaraswamy, Unit-Gamma and Log-Lindley distributions among others do not have
a stochastic representation as in \eqref{main-id}.

The aim of this paper is to propose an easy way of deriving the exact pdf of $X$ that is
involved in the ratio $Z = X/(X+Y)$ when $X$ and $Y$ are independent random variables.  The random variables $X$ and $Y$ do not need to belong to the same family of distributions.
We emphasize here that the techniques used in this work can be slightly modified to determine the distribution of $X$ in the independent ratio $Z=X/Y$, which for sake of conciseness are omitted here.
As the stochastic representation in \eqref{main-id} is important, since it may justify some models arising naturally in certain real situations, we can use the proposed approach to find the stochastic representation in \eqref{main-id} for several models known in the literature.
For example, it allows us to develop an EM-algorithm for estimating the parameters of the distribution of $Z$. We
further propose three new models for bounded data and study them in detail.
The rest of this paper proceeds as follows.
In Section \ref{Main results and examples}, we develop the main results and  study some special cases in detail. Then, some brief closing remarks are made in Section \ref{sec:3}.

\newpage 
\section{Main results and examples}\label{Main results and examples}
\noindent

Suppose 
%\begin{align*}
%\mathbb{P}\left(Z\leqslant z\right)
%=
%\mathbb{P}\left(X(1-z)\leqslant zY\right)
%=
%\mathbb{P}\left(X/Y\leqslant z/(1-z)\right)
%=
%\int_{0}^\infty f_Y(y) F_X([z/(1-z)]y){\rm d}y
%\end{align*}
%$
%\int \mathbb{P}(x/Y\leqslant 1/(1/z-1)) f_X(x)dx
%=
%\int \mathbb{P}(Y\geqslant  x(1/z-1)) f_X(x)dx
%$
%
$X$ and $Y$ are independent random variables, and that
$Y$ and $Z$ have known distributions such that the probability density function (PDF) of $Y$ admits the following decomposition:
	\begin{align}\label{decomp-pdf-Y}
f_Y(sx)
=
\mathbbm{A}(s) \mathbbm{B}(x) \mathbbm{C}(sx), \quad x, s>0,
\end{align}
where $\mathbbm{A}, \mathbbm{B}$ and $\mathbbm{C}$ are some positive-real functions.
Then, the main problem addressed here is in developing
mathematical tools for finding the distribution
of $X$ for a wide class of distributions.

\begin{theorem1}\label{tehorem-main-1-6}
	Under the conditions \eqref{main-id} and \eqref{decomp-pdf-Y}, if
	\begin{align}\label{decomp-pdf-Y-0-6}
		\mathbbm{C}(x)=\exp(-\lambda x^\theta), \quad x>0, \ \lambda,\theta >0,
	\end{align}
	the density of $X$ is given by
	\begin{align*}
		f_X(x)
		=
		\dfrac{\lambda\theta}{x^{2-\theta} \mathbbm{B}(x)}\,
		\mathcal{L}^{-1}\left\{	{1\over \mathbbm{A}(s^{1/\theta}) (s^{1/\theta}+1)^2}\,
		f_Z\left({1\over s^{1/\theta}+1}\right)\right\}(\lambda x^\theta),
	\end{align*}
	where $\mathcal{L}^{-1}$ is the inverse Laplace transform.
\end{theorem1}

From here on, in all the examples to follow,
we suppose that $X,Y$ and $Z$ are related through \eqref{main-id}, and that $X$ and $Y$ are independent.

\begin{example1}
 When $Y\sim\exp(\lambda)$, $\lambda>0$, and $Z\sim {\rm Kumaraswamy}(a,b)$, $a>0, b=1,2,\ldots$, that is, $Z$ has PDF given by
	\begin{align}\label{Kumaraswamy}
	{\displaystyle
	f_Z(z)
	=
	abz^{a-1}{(1-z^{a})}^{b-1},
	\quad 0<z<1.
	}
	\end{align}
	
%	In what follows we find the distribution of $X$.

Indeed, since $f_Y(sx)=\mathbbm{A}(s) \mathbbm{B}(x) \mathbbm{C}(sx)$, with
\begin{align*}
	\mathbbm{A}(s)=1,
	\quad
	\mathbbm{B}(x)={\lambda}\,
	\quad
	\text{and}
	\quad
	\mathbbm{C}(sx)=\exp(-\lambda sx),
\end{align*}
from Theorem \ref{tehorem-main-1-6} (with $\theta=1$), we readily have 
\begin{align}\label{last-ex}
	f_X(x)
	=
	{1\over x} \,
	\mathcal{L}^{-1}\left\{{1\over (s+1)^2}\, f_Z\left({1\over s+1}\right)\right\}(\lambda x).
\end{align}

On the other hand, a binomial expansion provides
$
f_Z(z)
=
ab
\sum_{k=0}^{b-1} \binom{b-1}{k} (-1)^k z^{a(k+1)-1}.
$
Using this expansion and the linearity of $\mathcal{L}^{-1}$, we can write \eqref{last-ex} as  
\begin{align}\label{rhs-fX}
	f_X(x)
	=
	{ab}\,
	{1\over x}
	\sum_{k=0}^{b-1}\binom{b-1}{k} (-1)^k
	\mathcal{L}^{-1}\left\{\left({1\over s +1}\right)^{a(k+1)+1}\right\}(\lambda x).
\end{align}
Now, by employing the
%formula:
%$\mathcal{L}^{-1}\{(s+1)^{-b}\}(t)=t^{b-1} \exp(-t)/ \Gamma(b)$
%%$\mathcal{L}^{-1}\{\Gamma(\nu)(s+a)^{-\nu}\}(t)=t^{-\nu-1}\exp(-at)$,
%%  $\nu>0$,
%(see Proposition \ref{prop-app-1})
%
well-known formula \cite[see][]{Erdelyi1954b}:
\begin{align}\label{usual-identity}
\mathcal{L}^{-1}\{(\alpha s+\beta)^{-p}\}(t)
=
\dfrac{1}{\alpha\Gamma(p)}\,
\left({t\over \alpha}\right)^{p-1}
\exp\left(-{\beta t\over \alpha}\right), \quad p>0,
\end{align}
the right-hand side of \eqref{rhs-fX} is	
\begin{align*}	
f_X(x)	&
	=
	{ab}\,
	{1\over x}
	\sum_{k=0}^{b-1}\binom{b-1}{k} (-1)^k \,
	{(\lambda x)^{a(k+1)}\over \Gamma(1 + a(k+1))} \, \exp(-\lambda x).
\end{align*}		
Hence, the density of $X$ in this case is given by	
\begin{align*}		
	f_X(x)
	&=
	{ab\lambda^{a(k+1)}}
	\sum_{k=0}^{b-1}\binom{b-1}{k} (-1)^k \,
	{x^{a(k+1)-1}\over \Gamma(1 + a(k+1))} \, \exp(-\lambda x).
	%	
	%	{1\over \Gamma(a(k+1)+1)}\, (\lambda x)^{a(k+1)} \exp(-\lambda x)
	%		\\[0,2cm]
	%	&=
	%	ab
	%	\sum_{k=0}^{b-1}\binom{b-1}{k} (-1)^k \,
	%	{\lambda^{a(k+1)}\over \Gamma(a(k+1)+1)}\, x^{a(k+1)-1} \exp(-\lambda x)
	\\[0,2cm]
	&=
\sum_{k=0}^{b-1}\binom{b-1}{k} (-1)^k\,
{1\over k+1}\, f_{T_k}(x), \quad x>0,
\end{align*}
	where $T_k\sim{\rm Gamma}(a(k+1),\lambda)$. Thus, the PDF of $X$ is a finite sum of weighted gamma distributions.
%
%
%Therefore,
%\begin{align*}
%	f_X(x)
%	&=
%	ab
%	\sum_{k=0}^{b-1}\binom{b-1}{k} (-1)^k \,
%	{\lambda^{a(k+1)}\over \Gamma(a(k+1)+1)}\, x^{a(k+1)-1} \exp(-\lambda x)
%	\\[0,2cm]
%	&=
%	ab
%	\sum_{k=0}^{b-1}\binom{b-1}{k} (-1)^k\,
%	{\Gamma(a(k+1))\over \Gamma(a(k+1)+1)}\, f_{T_k}(x)
%	\\[0,2cm]
%	&=
%	b
%	\sum_{k=0}^{b-1}\binom{b-1}{k} (-1)^k\,
%	{1\over (k+1)}\, f_{T_k}(x),
%\end{align*}
%where $T_k\sim{\rm Gamma}(a(k+1),\lambda)$.

%
%\begin{align*}
%\int_0^\infty
%	f_X(x)
%	{\rm d}x
%=
%b
%\sum_{k=0}^{b-1}\binom{b-1}{k} (-1)^k
%{1\over k+1}
%=
%1.
%\end{align*}
\end{example1}

%\begin{example1}
%Let $Y\sim{\rm Gamma}(b,1)$ and $Z\sim{\rm Beta}(a,b)$, $a,b>0$.
%
%In what follows we find the distribution of $X$.
%
%Indeed, since $f_Y(sx)=\mathbbm{A}(s) \mathbbm{B}(x) \mathbbm{C}(sx)$, with
%\begin{align*}
%\mathbbm{A}(s)= s^{b-1},
%\quad
%\mathbbm{B}(x)={1\over \Gamma(b)}\, x^{b-1}
%\quad
%\text{and}
%\quad
%\mathbbm{C}(sx)=\exp(-sx),
%\end{align*}
%from Theorem \ref{tehorem-main-1-6} (with $\lambda=\theta=1$) we have
%\begin{align*}
%f_X(x)
%&=
%\Gamma(b)\,
%{1\over  x^b}\,
%\mathcal{L}^{-1}\left\{{1\over s^{b-1}(s+1)^2}\, f_Z\left({1\over s+1}\right)\right\}(x)
%\\[0,2cm]
%&=
%{\Gamma(b)\over B(a,b)}\,
%{1\over  x^b}\,
%\mathcal{L}^{-1}\left\{(s+1)^{-(a+b)}\right\}(x)
%=
%{1\over \Gamma(a)}\, x^{a-1} \exp(-x),
%\end{align*}
%where in the last equality we used the well-known formula \cite[see][]{Erdelyi1954b}:
%\begin{align}\label{usual-identity}
%\mathcal{L}^{-1}\{(\alpha s+\beta)^{-p}\}(t)
%=
%\dfrac{1}{\alpha\Gamma(p)}\,
%\left({t\over \alpha}\right)^{p-1}
%\exp\left(-{\beta t\over \alpha}\right), \quad p>0.
%\end{align}
%In other words, $X\sim{\rm Gamma}(a,1)$. This confirms a well-known result in the literature with respect to type I radio involving independent beta variables.
%\end{example1}

\begin{example1}\label{bbeta}
When $Y\sim{\rm Gamma}(\beta,\lambda)$, $\lambda>0$, and $Z\sim{\rm Bbeta}(\alpha,\beta, \rho, \delta)$, $\alpha,\beta>0$, $\rho\geqslant 0$ and $\delta\in\mathbb{R}$, is a random variable following the bimodal  beta (Bbeta) distribution \cite[see][]{Vila:22} with density
\begin{align*}
    f_Z(z)
=
\displaystyle
\frac{\rho+(1-\delta{x})^2}{K {B}(\alpha,\beta) } \,
x^{\alpha-1} \, (1-x)^{\beta-1},
\quad 0 < z < 1,
\end{align*}
where
$
K
=
1+\rho
- 2\delta{\alpha/(\alpha+\beta)}
+
\delta^2{\alpha(\alpha+1)/[(\alpha+\beta)(\alpha+\beta+1)]}.
$
%
%In what follows we find the distribution of $X$.
%
Indeed, since $f_Y(sx)=\mathbbm{A}(s) \mathbbm{B}(x) \mathbbm{C}(sx)$, with
\begin{align*}
\mathbbm{A}(s)=s^{\beta-1},
\quad
\mathbbm{B}(x)={\lambda^\beta\over \Gamma(\beta)}\, x^{\beta-1}
\quad
\text{and}
\quad
\mathbbm{C}(sx)=\exp(-\lambda sx),
\end{align*}
from Theorem \ref{tehorem-main-1-6} (with $\theta=1$), we readily have
	\begin{align}\label{applic-theorem}
f_X(x)
=
\dfrac{\Gamma(\beta)}{\lambda^{\beta-1} x^{\beta}}\,
\mathcal{L}^{-1}\left\{	{1\over s^{\beta-1} (s+1)^2}\,
f_Z\left({1\over s+1}\right)\right\}(\lambda x),
\end{align}
where
\begin{align}\label{dec-frac}
{1\over s^{\beta-1} (s+1)^2}\,
f_Z\left({1\over s+1}\right)
=
\sum_{k=0}^{2}
{\pi_k\over\Gamma(\alpha+\beta+k)}\, {(s+1)^{-(\alpha+\beta+k)}},
\end{align}
$
\pi_0={(1+\rho)/K},
\pi_1=-{2\alpha \delta/[(\alpha+\beta) K]}
$
and
$
\pi_2={\alpha(\alpha+1)\delta^2/[(\alpha+\beta)(\alpha+\beta+1) K]}.
$
Note that $\pi_0+\pi_1+\pi_2=1$. Using \eqref{dec-frac} in \eqref{applic-theorem},
%, from identity \eqref{usual-identity}
we obtain
\begin{align*}
f_X(x)
&=
\dfrac{\Gamma(\beta)}{\lambda^{\beta-1} x^{\beta}}\,
\sum_{k=0}^{2}
{\pi_k\over B(\alpha+k,\beta)}\,
\mathcal{L}^{-1}\left\{ {(s+1)^{-(\alpha+\beta+k)}}\right\}(\lambda x)
\\[0,2cm]
&=
\sum_{k=0}^{2}
{\pi_k}\,
{\lambda^{\alpha+k}\over\Gamma(\alpha+k)}\,
x^{\alpha+k-1}
\exp(-\lambda x),
\end{align*}
where, in the last equality, we have used \eqref{usual-identity}.
Thus, the PDF of $X$ in this case can be written as a finite (generalized) mixture of three Gamma distributions with different shape parameters.
\end{example1}

\begin{remark1}
	Notice that the PDF of $Z$ in Example \ref{bbeta} can be written as
	\begin{align*}
	f_Z(z)=\pi_0 f_{Z_0}(z)+\pi_1 f_{Z_1}(z)+\pi_2 f_{Z_2}(z),
	\end{align*}
	where $Z_k\sim{\rm Beta}(\alpha+k,\beta)$, $k=0,1,2$, and $\pi_0, \pi_1, \pi_2$ being as defined in Example \ref{bbeta}.
\end{remark1}

Upon setting $\delta=0$ in Example \ref{bbeta}, the following result well-known in the literature is deduced.
\begin{example1}[\citet{malik:67,ahuja:69,Jones2021}]
		If $Y\sim{\rm Gamma}(\beta,\lambda)$, $\lambda>0$, and $Z\sim{\rm Beta}(\alpha,\beta)$, $\alpha,\beta>0$,
		then $X\sim{\rm Gamma}(\alpha,\lambda)$.
\end{example1}

\begin{example1}
	When $Y\sim{\rm Gamma}(\beta,\lambda)$, $\beta,\lambda>0$, and $Z$ is a random variable following the Topp-Leone distribution 
%\cite[][p. 317]{Nadarajah2003}, i.e.,
	with density
	\begin{align*}
	f_Z(z)
	=
	2vz^{v-1}(1-z)(2-z)^{v-1},
	\quad 0 < z < 1, \ v=1,2,\ldots.
	\end{align*}
%	
%		In what follows we find the distribution of $X$.
	Indeed, by taking 	$\mathbbm{A}(s),
	\mathbbm{B}(x)$ and
	$\mathbbm{C}(sx)$ as in Example \ref{bbeta}, from Theorem \ref{tehorem-main-1-6} (with $\theta=1$), we have the validity of the identity in \eqref{applic-theorem}, with $Z$ following the Topp-Leone distribution. Note that this identity can be expressed equivalently as
	\begin{align*}
f_X(x)
=
\dfrac{2v\Gamma(\beta)}{\lambda^{\beta-1} x^{\beta}}\,
\mathcal{L}^{-1}\left\{	{(2s+1)^{v-1}\over s^{\beta-2} (s+1)^{2v+1}}\right\}(\lambda x).
\end{align*}
Using a  binomial expansion, the expression on the right hand side becomes
	\begin{multline*}
\dfrac{2v\Gamma(\beta)}{\lambda^{\beta-1} x^{\beta}}
\sum_{k=0}^{v-1}
\binom{v-1}{k}
2^k
\mathcal{L}^{-1}\left\{	{1\over s^{\beta-k-2} (s+1)^{2v+1}}\right\}(\lambda x)
\\[0,2cm]
=
{2v\Gamma(\beta)}
\sum_{k=0}^{v-1}
\binom{v-1}{k}
2^k \lambda^{2v-k-1} x^{2v-k-2} \, {_1F_1(2v+1;\beta+2v-k-1;-\lambda x)\over\Gamma(\beta+2v-k-1)},
	\end{multline*}
	where, in the last line, Proposition \ref{prop-app-1} has been used. Thus, the PDF of $X$ is a finite sum in this case of weighted  distributions that include confluent hypergeometric functions.
\end{example1}

	\begin{example1}\label{weighted Lindley}
	When $Y\sim\exp(\lambda)$, $\lambda>0$, and $Z$ is a random variable having the density
\begin{align*}
f_Z(z)
=
{\lambda\beta^{c+1}\over (\beta+c)\Gamma(c)}\, {1\over z^2}
\left\{
\Gamma(c+1) \left[\lambda \left({1\over z}-1\right)+\beta \right]^{-(c+1)}
+
\Gamma(c+2) \left[\lambda \left({1\over z}-1\right)+\beta \right]^{-(c+2)}
\right\},
\end{align*}
where
%$s=(1/ z)-1$,
$0<z<1$ and $c,\beta>0$.
%
%In what follows we find the distribution of $X$.
%
Indeed, since $f_Y(sx)=\mathbbm{A}(s) \mathbbm{B}(x) \mathbbm{C}(sx)$, with
\begin{align*}
\mathbbm{A}(s)=1,
\quad
\mathbbm{B}(x)={\lambda}
\quad
\text{and}
\quad
\mathbbm{C}(sx)=\exp(-\lambda sx),
\end{align*}
from Theorem \ref{tehorem-main-1-6} (with $\theta=1$), we readily have
\begin{align*}
	f_X(x)
	&=
	{1\over  x}\,
	\mathcal{L}^{-1}\left\{{1\over (s+1)^2}\, f_Z\left({1\over s+1}\right)\right\}(\lambda x)
	\\[0,2cm]
	&=
	{\lambda\beta^{c+1}\over (\beta+c)\Gamma(c)} \,
	{1\over  x}\,
	\mathcal{L}^{-1}\left\{
	\Gamma(c+1)
(\lambda s+\beta)^{-(c+1)}
+
\Gamma(c+2) (\lambda s+\beta)^{-(c+2)}
	\right\}(\lambda x).
\end{align*}
Upon using \eqref{usual-identity},  the last equation becomes
\begin{align*}
f_X(x) =
{\beta^{c+1}\over (\beta+c)\Gamma(c)}\, x^{c-1}(1+x)\exp(-\beta x), \quad x>0;
\end{align*}
that is, $X$ is a random variable having the weighted Lindley distribution \cite[][]{Ghi:11}.
\end{example1}

\begin{remark1}
Note that the density of $Z$ in Example \ref{weighted Lindley} can be expressed as
\begin{align*}
f_Z(z)=pf_{T_0}(z)+(1-p)f_{T_1}(z),
\end{align*}
where $p=\beta/(\beta+c)$ and
% $f_{T_j}(z)=f_{L_j}\big(a(z);c+j,\theta/\lambda\big) \vert a'(z)\vert$, $j=0,1$,
% with
% $L_j\sim{\rm Lomax}(c+j,\theta/\lambda)$ and $a(z)=(1/z)-1$. I.e.,
 $T_j
% =a^{-1}(L_j)
 =1/(L_j+1)$,
  $L_j\sim{\rm Lomax}(c+j,\beta/\lambda)$,
  $j=0,1$. Further,
  \begin{align*}
  \mathbb{E}(T_j^r)={c+j\over c+j-r+4}\, \left({\beta\over\lambda}\right)^{c+j}\,
  _2F_1\left(c+j+1, c+j-r+4; c+j-r+5; 1-{\beta\over\lambda}\right),
  \end{align*}
  provided $c+j-r+4>0$.
\end{remark1}

%	The next result represents a generalization of Theorem \ref{tehorem-main-1-5}.

\begin{example1}\label{example-GG}
	When $Y\sim{\rm GG}(a_2,d_2,\theta)$, $a_2,d_2,\theta>0$, is a random variable following the generalized gamma (GG) distribution with density
	\begin{align*}
	f_Y(y)
	=
	\dfrac{\theta}{a_2^{d_2}\Gamma\left(\dfrac{d_2}{\theta}\right)}\, y^{d_2-1}\exp\left[-\left({y\over a_2}\right)^\theta\right], \quad y>0,
	\end{align*}
	and let  $Z$ is a random variable with density
	\begin{align}\label{LN}
	{\displaystyle
		f_Z(z)
		=
		\dfrac{\theta a_1^{d_2} a_2^{d_1}}{ B\left(\dfrac{d_1}{\theta},\dfrac{d_2}{\theta}\right)}\,
		\dfrac{z^{d_1-1}(1-z)^{d_2-1}}{\{(a_2 z)^\theta+[a_1(1-z)]^\theta\}^{(d_1+d_2)/\theta}},
		\ \ {\mbox{where}}\ \ 0<z<1, \ a_1,d_1>0.
	}
	\end{align}
	
	%In what follows we find the distribution of $X$.
	
	Indeed, since $f_Y(sx)=\mathbbm{A}(s) \mathbbm{B}(x) \mathbbm{C}(sx)$, with
	\begin{align*}
	\mathbbm{A}(s)=s^{d_2-1},
	\quad
	\mathbbm{B}(x)={\theta\over a_2^{d_2}\Gamma\left(\dfrac{d_2}{\theta}\right)}\,  x^{d_2-1}
	\quad
	\text{and}
	\quad
	\mathbbm{C}(sx)=\exp\left[- \left({s x\over a_2}\right)^\theta\right],
	\end{align*}
	from Theorem \ref{tehorem-main-1-6} (with $\lambda=a_2^{-\theta}$), we readily have 
	\begin{align*}
	f_X(x)
	=
	\dfrac{\lambda\theta}{x^{2-\theta} \mathbbm{B}(x)}\,
	\mathcal{L}^{-1}\left\{	{1\over \mathbbm{A}(s^{1/\theta}) (s^{1/\theta}+1)^2}\,
	f_Z\left({1\over s^{1/\theta}+1}\right)\right\}(\lambda x^\theta).
	\end{align*}
	But, $x^{2-\theta} \mathbbm{B}(x)=
	\theta x^{d_2-\theta+1}/[a_2^{d_2}\Gamma(d_2/\theta)]$ and
	\begin{align*}
	{1\over \mathbbm{A}(s^{1/\theta}) (s^{1/\theta}+1)^2}\,
	f_Z\left({1\over s^{1/\theta}+1}\right)
	=
		{\theta  a_2^{d_1}\over a_1^{d_1} B\left(\dfrac{d_1}{\theta},\dfrac{d_2}{\theta}\right)}\,
\left[\left(\dfrac{a_2}{a_1}\right)^\theta+s\right]^{-(d_1+d_2)/\theta}.
	\end{align*}
Then, by using \eqref{usual-identity}, we find
	\begin{align*}
	f_X(x)
	&=
	\dfrac{\lambda\theta a_2^{d_1+d_2}\Gamma\left(\dfrac{d_2}{\theta}\right)}{ a_1^{d_1} B\left(\dfrac{d_1}{\theta},\dfrac{d_2}{\theta}\right) }\,
	{1\over x^{d_2-\theta+1}}\,
	\mathcal{L}^{-1}\left\{
\left[\left(\dfrac{a_2}{a_1}\right)^\theta+s\right]^{-(d_1+d_2)/\theta}
	\right\}(\lambda x^\theta)
	\\[0,2cm]
	&=
	\dfrac{\theta}{a_1^{d_1} \Gamma\left(\dfrac{d_1}{\theta}\right)}\,
		x^{d_1-1}\,
	\exp\left[- \left(\dfrac{x}{a_1}\right)^\theta\right], \quad x>0;
	\end{align*}
	that is, $X\sim{\rm GG}(a_1,d_1,\theta)$. Note that, by taking $\theta=1$, the PDF of $Z$ in \eqref{LN} becomes the Libby-Novick distribution \citep[see][]{Ahmed21, LN82}. Also, by taking $d_1=d_2=\theta$, $a_1=1$ and $a_2=\beta$ in Example \ref{example-GG}, we get the following example.
\end{example1}

\begin{example1}
	When $Y\sim{\rm Weibull}(\theta,\beta)$, $\theta,\beta>0$, and $Z\sim {\rm UW2}(\theta,\beta)$ is a random variable following a unitary Weibull distribution Type 2 (UW2) \citep[see][]{Reyes23} with density
	\begin{align*}
	{\displaystyle
		f_Z(z)
		=
		\dfrac{\theta\beta^\theta z^{\theta-1}(1-z)^{\theta-1}}{[(\beta z)^\theta+(1-z)^\theta]^2},
		\quad 0<z<1.
	}
	\end{align*}
	Then, $X\sim{\rm Weibull}(\theta,1)$.
%
%	In what follows we find the distribution of $X$.
%	
%	Indeed, since $f_Y(sx)=A(s) B(x) C(sx)$ with $A(s)=s^{\theta-1}$, $B(x)=(\theta/\beta)(x/\beta)^{\theta-1}$ and $C(sx)=\exp(- (s x)^\theta/\beta^\theta)$, from Theorem \ref{tehorem-main-1-6} (with $\lambda=\beta^{-\theta}$) we have
%	\begin{align*}
%	f_X(x)
%	=
%	\dfrac{\lambda\theta}{x^{2-\theta} B(x)}\,
%	\mathcal{L}^{-1}\left\{	{1\over A(s^{1/\theta}) (s^{1/\theta}+1)^2}\,
%	f_Z\left({1\over s^{1/\theta}+1}\right)\right\}(\lambda x^\theta).
%	\end{align*}
%	But $x^{2-\theta} B(x)=(\theta/\beta^{\theta}) x$ and
%	\begin{align*}
%	{1\over A(s^{1/\theta}) (s^{1/\theta}+1)^2}\,
%	f_Z\left({1\over s^{1/\theta}+1}\right)
%	=
%	{\theta\beta^\theta
%		\over (\beta^\theta+s)^2}.
%	\end{align*}
%	Furthermore, it is well-known that $\mathcal{L}^{-1}\{(a+s)^{-p}\}(x)=x^{p-1}\exp(-a x)/\Gamma(p)$, $p>0$. Then
%	\begin{align*}
%	f_X(x)
%	=
%	\dfrac{1}{x}\,
%	\mathcal{L}^{-1}\left\{{\theta\beta^\theta
%		\over (\beta^\theta+s)^2}\right\}(\lambda x^\theta)
%	=
%	\theta x^{\theta-1} \exp(-x^\theta), \quad x>0.
%	\end{align*}
%	In other words, $X\sim{\rm Weibull}(\theta,1)$. This confirms the known result in \cite{Reyes23}.
\end{example1}

\begin{theorem1}\label{tehorem-main-1-7}
	Under the conditions in \eqref{main-id} and \eqref{decomp-pdf-Y}, if
	\begin{align}\label{decomp-pdf-Y-0-7}
	\mathbbm{C}(x)=(p+qx)\exp(-\lambda x), \quad x>0, \ q,\lambda>0, p\geqslant 0,
	\end{align}
	the density of $X$ is given by
\begin{align*}
	f_X(x)
	=
	{\lambda^3\over qx^{(\lambda p/ q)+2} \mathbbm{B}(x)}
\,
	\int_0^x
	\xi^{\lambda p/ q}
	\mathcal{L}^{-1}\left\{
	{1\over \mathbbm{A}\left(\dfrac{s}{\lambda}\right) (s+\lambda)^2}\,
	f_Z\left({\lambda\over s+\lambda}\right)
	\right\}(\xi)
	{\rm d}\xi,
\end{align*}
whenever $[x^{(\lambda p/ q)+2} \mathbbm{B}(x) f_X(x)]\vert_{x=0^+}=0$.
\end{theorem1}

\begin{example1}\label{weighted Lindley distribution}
	When $Y$ is a random variable following the weighted Lindley distribution \cite[see][]{Ghi:11} with parameter vector $(b,1)$ and density
	\begin{align*}
	f_Y(y)
	&=
	{1\over (1+b)\Gamma(b)}\, y^{b-1}(1+y)\exp(-y), \quad y>0, \ b>0,
	\end{align*}
and $Z$ is a random variable with density
\begin{align}\label{misture-betas}
f_Z(z)
=
	{(a+b+1)\over (1+a)(1+b) B(a,b)}\,
	{z^{a-1}}  \left({1-z}\right)^{b-1}
	\left[(a+b)\, z{(1-z)}+ {1}\right], 	\quad 0<z<1, \ a>0.
\end{align}

%In what follows we find the distribution of $X$.

	Indeed, since $f_Y(sx)=\mathbbm{A}(s) \mathbbm{B}(x) \mathbbm{C}(sx)$, with
	\begin{align*}
	\mathbbm{A}(s)=s^{b-1},
	\quad
	\mathbbm{B}(x)={1\over (1+b)\Gamma(b)}\, x^{b-1}
	\quad
	\text{and}
	\quad
	\mathbbm{C}(sx)=(1+sx)\exp(-sx),
	\end{align*}
	from Theorem \ref{tehorem-main-1-7} (with $p=q=\lambda=1$), we get	
\begin{align}\label{eq:1}
f_X(x)
=
{1\over x^{3} \mathbbm{B}(x)}
\int_0^x
\xi
\mathcal{L}^{-1}\left\{
{1\over \mathbbm{A}(s) (s+1)^2}\,
f_Z\left({1\over s+1}\right)
\right\}(\xi)
{\rm d}\xi,
\end{align}
provided $[x^3 \mathbbm{B}(x) f_X(x)]\vert_{x=0^+}=0$.
But, $x^{3} \mathbbm{B}(x)=x^{b+2}/[(1+b)\Gamma(b)]$ and
\begin{align*}
{1\over \mathbbm{A}(s) (s+1)^2}\,
f_Z\left({1\over s+1}\right)
&=
{(a+b+1)\over (1+a)(1+b) B(a,b)}
\left[
{1\over (s+1)^{a+b}}
+
{(a+b)} \,
{s\over (s+1)^{a+b+2}}
\right]
\\[0,2cm]
&=
{(a+b+1)\over (1+a)(1+b) B(a,b)}
\left[{1 \over (s+1)^{a+b}}+(a+b)s\mathcal{L}\left\{{\xi^{a+b+1}\exp(-\xi)\over\Gamma(a+b+2)}\right\}(s) \right],
\end{align*}
where, in the last line, we have used \eqref{usual-identity}. Now, by applying the inverse Laplace transform on both sides of the above equality and using \eqref{usual-identity} together with property \eqref{property-inv-Laplace}, we obtain
\begin{multline*}
\mathcal{L}^{-1}\left\{
{1\over \mathbbm{A}(s) (s+1)^2}\,
f_Z\left({1\over s+1}\right)
\right\}(\xi)
\\[0,2cm]
=
{(a+b+1)\over (1+a)(1+b) B(a,b)}
\left[
\mathcal{L}^{-1}\left\{(s+1)^{-(a+b)}\right\}(\xi)
+
(a+b)
\mathcal{L}^{-1}\left\{
s\mathcal{L}\left\{{\xi^{a+b+1}\exp(-\xi)\over\Gamma(a+b+2)}\right\}(s)
\right\}(\xi)
\right]
\\[0,2cm]
=
{1\over (1+a)(1+b)\Gamma(a)\Gamma(b)}\,
\xi^{a+b-1}[(a+b+1)+(a+b+1)\xi-\xi^2]\exp(-\xi).
\end{multline*}
Using the above identities in \eqref{eq:1}, we get
\begin{align*}
f_X(x)
&=
{1\over (1+a)\Gamma(a)}\,
{1\over x^{b+2}} \,
\int_0^x
\xi^{a+b}[(a+b+1)+(a+b+1)\xi-\xi^2]\exp(-\xi)
{\rm d}\xi
\\[0,2cm]
&=
{1\over (1+a)\Gamma(a)}\,
x^{a-1}(1+x)\exp(-x),
\end{align*}	
whenever $[x^3 \mathbbm{B}(x) f_X(x)]\vert_{x=0^+}=0$. Note that this condition is equivalent to  $[x^{b+2}f_X(x)]\vert_{x=0^+}=0$, which is satisfied if we consider $f_X$ as in the above equation.
Then, we conclude that $X$ has weighted Lindley distribution with parameter vector $(a,1)$.
\end{example1}	

\begin{remark1}
	A simple observation shows that the density of $Z$ in \eqref{misture-betas} can be written as a mixture of beta distributions as 
	\begin{align*}
	f_Z(z)
	=
	p f_{Z_0}(z)+(1-p) f_{Z_1}(z),
	\end{align*}
	where $p=(a+b+1)/[(a+1)(b+1)]$ and $Z_j\sim {\rm Beta}(a+j,b+j)$, $j=0,1$.
\end{remark1}

%It is not always easy to obtain an explicit inverse kernel $K^{-1}$ in the representation \eqref{identity-main} of the PDF of $X$ (Theorem \ref{main-theorem}). Instead,
%In many situations, it is possible to obtain other representations of \eqref{identity-main} as a function of known transforms. This is reflected in the following result.
%
\begin{theorem1}\label{tehorem-main-2}
	Under the conditions in \eqref{main-id} and \eqref{decomp-pdf-Y}, if
\begin{align}\label{decomp-pdf-Y-1}
\mathbbm{C}(x)=(1+\theta x)^{-p}, \quad x>0,\ \theta, p>0,
\end{align}
the density of $X$ is given by
\begin{align*}
f_X(x)
=
{1\over x \mathbbm{B}(x)}\,
\mathcal{G}_p^{-1}\left\{
{1\over s^{p} \mathbbm{A}\left(\dfrac{1}{\theta s}\right)}
\left({\theta s\over 1+\theta s}\right)^2
f_Z\left({\theta s\over 1+ \theta s}\right)
\right\}(x),
\end{align*}
where $\mathcal{G}_p^{-1}$ is the inverse of the generalized Stieltjes transform $\mathcal{G}_p$ \citep[see][]{Schwarz05}.
%
%\begin{align*}
%	f_X(x)
%	=
%	{\Gamma(p) \over x B(x) }\,
%	\mathcal{L}^{-1}\left\{
%	t^{1-p}
%	\mathcal{L}^{-1}\left\{
%	{1\over s^{p} A\left(\dfrac{1}{\theta s}\right) }
%	\left({\theta s\over 1+\theta s}\right)^2
%	f_Z\left({\theta s\over 1+ \theta s}\right)
%	\right\}(t)
%	\right\}(x).
%\end{align*}
\end{theorem1}

\begin{example1}
	When $Y$ is a random variable having the generalized beta-prime distribution with shape parameters $\alpha_2, \beta_2>0$ and density
	\begin{align*}
	f_Y(y)
	=
	{\lambda_2^{\alpha_2} y^{\alpha_2-1} (1+\lambda_2 y)^{-(\alpha_2+\beta_2)}\over B(\alpha_2,\beta_2) },
	\quad y>0,
	\end{align*}
and $Z$ has density
%\begin{align*}
%	f_Z(z)
%	=
%K
%	(1-z)^{\alpha_2-1} z^{-(1+\alpha_2)}
%	\, _2F_1\left(\alpha_2+\beta_2,\alpha_1+\alpha_2;\alpha_1+\alpha_2+\beta_1+\beta_2;1-{\lambda_2\over\lambda_1} \left({1\over z}-1\right)\right),
%\end{align*}
\begin{align}\label{pdf-Z}
f_Z(z)
=
K\,
{z^{\alpha_1-1}\over (1-z)^{\alpha_1+1}}
\, _2F_1\left(\alpha_1+\beta_1,\alpha_1+\alpha_2;\alpha_1+\alpha_2+\beta_1+\beta_2;1-{\lambda_1\over\lambda_2} \left({z\over 1-z}\right)\right), 	 \quad 0<z<1,
\end{align}
where
$K=B(\alpha_1+\alpha_2,\beta_1+\beta_2) \lambda_1^{\alpha_1}/[ B(\alpha_1,\beta_1)B(\alpha_2,\beta_2) \lambda_2^{\alpha_1}]
$ and  $\alpha_1, \beta_1>0$.
%
%In what follows we find the distribution of $X$.
%
	Indeed, observe that $f_Y(sx)=\mathbbm{A}(s) \mathbbm{B}(x) \mathbbm{C}(sx)$, with
\begin{align*}
	\mathbbm{A}(s)=s^{\alpha_2-1},
	\quad
	\mathbbm{B}(x)={\lambda_2^{\alpha_2}\over B(\alpha_2 ,\beta_2)}\, x^{\alpha_2-1}
	\quad
	\text{and}
	\quad
	\mathbbm{C}(sx)=(1+\lambda_2 sx)^{-(\alpha_2+\beta_2)}.
\end{align*}		
	
As $x \mathbbm{B}(x)=\lambda_2^{\alpha_2} x^{\alpha_2}/B(\alpha_2 ,\beta_2 )$ and (for $p=\alpha_2+\beta_2$ and $\theta=\lambda_2$)
	\begin{multline*}
	{1\over s^{p} \mathbbm{A}\left(\dfrac{1}{\theta s}\right)}
	\left({\theta s\over 1+\theta s}\right)^2
	f_Z\left({\theta s\over 1+ \theta s}\right)
	\\[0,2cm]
	=
	\dfrac{B(\alpha_1+\alpha_2,\beta_1+\beta_2) \lambda_1^{\alpha_1}\lambda_2^{\alpha_2}}{ B(\alpha_1,\beta_1)B(\alpha_2,\beta_2) } \,
	s^{\alpha_1-\beta_2}
	\, _2F_1\left(\alpha_1+\beta_1,\alpha_1+\alpha_2;\alpha_1+\alpha_2+\beta_1+\beta_2;1-\lambda_1 s\right),
	\end{multline*}
	from Theorem \ref{tehorem-main-2} (with $p=\alpha_2+\beta_2$ and $\theta=\lambda_2$), we readily obtain
	\begin{align}\label{id-1}
	f_X(x)
	=
	\dfrac{B(\alpha_1+\alpha_2,\beta_1+\beta_2) \lambda_1^{\alpha_1}}{ B(\alpha_1,\beta_1) x^{\alpha_2}} \,
	\mathcal{G}_p^{-1}\left\{
s^{\alpha_1-\beta_2}
\, _2F_1\left(\alpha_1+\beta_1,\alpha_1+\alpha_2;\alpha_1+\alpha_2+\beta_1+\beta_2;1-\lambda_1 s\right)
	\right\}(x).
	\end{align}
	By using the known formula \cite[see][p. 233]{Erdelyi1954b}
	\begin{align*}
	\mathcal{G}_\rho
	\left\{
	x^{\nu-1}(a+x)^{-\mu}
	\right\}(y)	
	=
	\dfrac{\Gamma(\nu)\Gamma(\mu-\nu+\rho) }{\Gamma(\mu+\rho) a^\mu}\,
	y^{\nu-\rho}\,
	_2F_1\left(\mu,\nu;\mu+\rho;1-{y\over a}\right), \quad \rho>\nu-\mu,
	\end{align*}
	with $\mu=\alpha_1+\beta_1, \nu=\alpha_1+\alpha_2, a=1/\lambda_1, \rho=p=\alpha_2+\beta_2$ and $y=s$, we write the argument of $\mathcal{G}_\rho^{-1}$ in \eqref{id-1} as 
\begin{multline}\label{id-2}
	s^{\alpha_1-\beta_2}
	\, _2F_1\left(\alpha_1+\beta_1,\alpha_1+\alpha_2;\alpha_1+\alpha_2+\beta_1+\beta_2;1-\lambda_1 s\right)
	\\[0,2cm]
	=
	\dfrac{\Gamma(\alpha_1+\beta_1+\alpha_2+\beta_2) }{\Gamma(\alpha_1+\alpha_2)\Gamma(\beta_1+\beta_2) \lambda_1^{\alpha_1+\beta_1}}\,
	\mathcal{G}_p
	\left\{
	x^{\alpha_1+\alpha_2-1}\left({1\over\lambda_1}+x\right)^{-(\alpha_1+\beta_1)}
	\right\}(s).
\end{multline}
Then, by combining \eqref{id-1} and \eqref{id-2}, we obtain
\begin{align*}
	f_X(x)
=
\dfrac{\lambda_1^{\alpha_1}
	x^{\alpha_1-1}
	\left(1+\lambda_1 x\right)^{-(\alpha_1+\beta_1)}}{ B(\alpha_1,\beta_1)};
\end{align*}
that is, $X$ follows the generalized beta-prime distribution with shape parameters $\alpha_1, \beta_1>0$.	
\end{example1}

\begin{remark1}
	For two independent variables $X$ and $Y$ having generalized beta-prime distributions, the density $f_Z$ equivalent to \eqref{pdf-Z} was determined earlier by \cite{Bekker2009}.
\end{remark1}

By combining the formula (23) of \cite{Schwarz05} with Theorem \ref{tehorem-main-2}, the following result follows.
\begin{corollary}\label{corollary}
	Under the conditions in \eqref{main-id} and \eqref{decomp-pdf-Y}, if $\mathbbm{C}(x)$ is as in \eqref{decomp-pdf-Y-1},
	then the density of $X$ is given by
	\begin{align*}
		f_X(x)
		=
		{\Gamma(p) \over x \mathbbm{B}(x) }\,
		\mathcal{L}^{-1}\left\{
		t^{1-p}
		\mathcal{L}^{-1}\left\{
		{1\over s^{p} \mathbbm{A}\left(\dfrac{1}{\theta s}\right) }
		\left({\theta s\over 1+\theta s}\right)^2
		f_Z\left({\theta s\over 1+ \theta s}\right)
		\right\}(t)
		\right\}(x).
	\end{align*}
%		where $\mathcal{G}_p^{-1}$ is the inverse of the generalized Stieltjes transform $\mathcal{G}_p$ \citep[see, e.g.,][]{Schwarz05}.
\end{corollary}

Corollary \ref{corollary} provides an alternative formula in case the inverse generalized Stieltjes transform of Theorem \ref{tehorem-main-2} is difficult to find.

\thispagestyle{empty}
%
%\begin{align*}
%f_Z(z)
%&=
%\int_0^\infty f_X(x) \exp\left(-\left({1\over z}-1\right)^\beta x^\beta\right) {\rm d}x
%\\[0,2cm]
%&=
%c
%\int_0^\infty
%\exp\left(- \alpha x\right)
%\exp\left(-\left({1\over z}-1\right)^\beta x^\beta\right)
% {\rm d}x
% \\[0,2cm]
% &=
%\alpha \mathcal{L}\left\{\exp\left(-\left({1\over z}-1\right)^\beta x^\beta\right)\right\}(\alpha).
%\end{align*}

%Another important result which can be directly applied when the PDF shape of Y is simple is the following.
%\begin{theorem}
%	
%\end{theorem}
%\begin{proof}
%	content
%\end{proof}

\section{Concluding Remarks} \label{sec:3}
\noindent

In the recent literature concerning bounded models, many papers have assumed a known distribution for $Z$, but usually it is difficult to understand how $X$ and $Y$ were obtained. In this paper, a new technique based on inverse integral transforms approach is suggested for finding the exact distribution of the random variable $X$ that is involved in the independent ratio $Z = X/(X+Y)$. This procedure has been discussed in detail for some cases and illustrated with many known as well as some new (see Examples \ref{weighted Lindley}, \ref{example-GG} and \ref{weighted Lindley distribution}) bounded models.

\paragraph{Acknowledgements}
\noindent

Roberto Vila gratefully acknowledge financial support from CNPq, CAPES and FAP-DF, Brazil.
Marcelo Bourguignon gratefully acknowledges partial financial support of the
Brazilian agency Conselho Nacional de Desenvolvimento Cient\'ifico e Tecnol\'ogico
(CNPq: grant 304140/2021-0).

\paragraph{Disclosure statement}
There are no conflicts of interest to disclose.

%%%%%%%%%%%%%%%%%%%%%%%%%%%%%%%%%%%%%%%%%%%%%%%%%%%%%%%%%%%%%

\normalsize

%\bibliographystyle{apalike}
%\bibliography{bibliography}

\begin{appendices}
	\section{A technical result}
%\section*{A technical result}
\noindent

\begin{proposition}\label{prop-app-1}
	The following identity holds true:
	\begin{align*}
	\mathcal{L}^{-1}\left\{{s^a\over (1+s)^b}\right\}(t)
	=
	{1\over \Gamma(b-a)}\, t^{b-a-1}\, _1F_1(b; b-a; -t), \quad b>a.
	\end{align*}
\end{proposition}
\begin{proof}
	Upon using the identity
	\begin{align*}%\label{id-F1}
	&_1F_1(a; c; z)={\Gamma(c)\over \Gamma(a)}\, \sum_{n=0}^{\infty} {\Gamma(a+n)\over \Gamma(c+n)}\, {z^n\over n!},
%	\\[0,2cm]
%	&_1F_1(a; c; -z)=\exp(-z)\, _1F_1(c-a; c; z),
	\end{align*}
	we can write (for  $b>a$)
	\begin{align*}
	{1\over \Gamma(b-a)}\, t^{b-a-1}\, _1F_1(b; b-a; -t)
	=
%	{1\over \Gamma(b-a)}\, t^{b-a-1}\, \exp(-t)\, _1F_1(-a; b-a; t).
	{1\over \Gamma(b)}\, \sum_{n=0}^{\infty} {\Gamma(b+n)\over \Gamma(b-a+n)}\, {(-1)^n\over n!}\,  {t^{n+b-a-1}}.
	\end{align*}
	Then,  by using Fubini/Tonelli theorem,
	\begin{align}
	\mathcal{L}\left\{
		{1\over \Gamma(b-a)}\, t^{b-a-1}\, _1F_1(b; b-a; -t)
		\right\}
	&=
	{1\over \Gamma(b)}\, \sum_{n=0}^{\infty} {\Gamma(b+n)\over \Gamma(b-a+n)}\, {(-1)^n\over n!}\,  \mathcal{L}\left\{t^{n+b-a-1}\right\}(s) \nonumber
	\\[0,2cm]
	&=
	s^{a-b}\,
	{1\over \Gamma(b)}\, \sum_{n=0}^{\infty} {\Gamma(b+n)}\, {(-1)^n\over n!} \,
	{1\over s^{n}}, \label{identity-gamma}
\end{align}
because $\mathcal{L}\{t^p\}(s)=\Gamma(p+1)/s^{p+1}$, $p>-1$. By combining the identity
\begin{align*}
\sum_{n=0}^{\infty} {\Gamma(b+n)}\, {(-1)^n\over n!} \,
{1\over s^{n}}
=
\left({s\over s+1}\right)^b \Gamma(b)
\end{align*}
with \eqref{identity-gamma}, the required result follows.
\end{proof}

\section{Proofs of the theorems}

\subsection*{Proof of Theorem 1}
\noindent

	Using the independence of $X$ and $Y$, from \eqref{main-id} and \eqref{decomp-pdf-Y}, it is simple to verify that the PDF of $Z$ can be expressed as
	%\begin{align*}
	%F_Z(z)
	%=
	%\int_0^\infty f_X(x) S_Y\left(\left({1\over z}-1\right)x\right) {\rm d}x,
	%\end{align*}
	\begin{align*}
		f_Z(z)
		=
		(s+1)^2
		\int_0^\infty x f_X(x) f_Y(sx) {\rm d}x,
		\quad {\mbox{where}}\  s={1\over z}-1.
	\end{align*}
	%Since $f_Y(sx)=A(s) B(x) C(sx)$, the above identity becomes
	%\begin{align}
	%f_Z(z)
	%=
	%A(s)  (s+1)^2
	%\int_0^\infty x B(x) f_X(x)  C(sx) {\rm d}x. \label{id-0}
	%\end{align}
	By using \eqref{decomp-pdf-Y}, the above identity can be equivalently written as
	\begin{align}\label{fZ-density}
		{1\over \mathbbm{A}(s) (s+1)^2}\,
		f_Z\left({1\over s+1}\right)
		=
		\int_0^\infty
		x \mathbbm{B}(x) f_X(x)
		\mathbbm{C}(sx)
		{\rm d}x.
	\end{align}
From \eqref{decomp-pdf-Y-0-6}, the above equation becomes
%	
%	
%	By combining  the decomposition in \eqref{decomp-pdf-Y} with \eqref{decomp-pdf-Y-0-6},
%	from identity \eqref{fZ-density} we have
	\begin{align*}
		{1\over \mathbbm{A}(s) (s+1)^2}\,
		f_Z\left({1\over s+1}\right)
		&=
		\int_0^\infty
		x \mathbbm{B}(x) f_X(x)
		\exp(-\lambda s^\theta x^\theta)
		{\rm d}x.
	\end{align*}
	Making the change of variable $y=x^\theta$, with ${\rm d}x=[y^{(1/\theta)-1}/\theta] {\rm d}y$, the above improper integral is
	\begin{align*}
		&=
		{1\over\theta}
		\int_0^\infty
		y^{(2/\theta)-1} \mathbbm{B}(y^{1/\theta}) f_X(y^{1/\theta})
		\exp(-\lambda s^\theta y)
		{\rm d}y \nonumber
		\\[0,2cm]
		&=
		{1\over\theta}\,
		\mathcal{L}\left\{
		y^{(2/\theta)-1}\mathbbm{B}(y^{1/\theta}) f_X(y^{1/\theta})
		\right\}(\lambda s^\theta)
		\\[0,2cm]
		&=
		{1\over\lambda\theta}\,
		\mathcal{L}\left\{\left({y\over\lambda}\right)^{(2/\theta)-1} \mathbbm{B}\left(\left({y\over\lambda}\right)^{1/\theta}\right)
		f_X\left(\left({y\over\lambda}\right)^{1/\theta}\right)
		\right\}(s^\theta),
	\end{align*}	
	where, in the last line, we have used the scale change property of $\mathcal{L}$. Hence,
	\begin{align*}
		{\lambda\theta\over \mathbbm{A}(s) (s+1)^2}\,
		f_Z\left({1\over s+1}\right)
		=
		\mathcal{L}\left\{\left({y\over\lambda}\right)^{(2/\theta)-1} \mathbbm{B}\left(\left({y\over\lambda}\right)^{1/\theta}\right)
		f_X\left(\left({y\over\lambda}\right)^{1/\theta}\right) \right\}(s^\theta).
	\end{align*}
	Now, by applying the inverse Laplace transform to both sides of the above equation, we obtain
	\begin{align*}
		\left({y\over\lambda}\right)^{(2/\theta)-1} \mathbbm{B}\left(\left({y\over\lambda}\right)^{1/\theta}\right)
		f_X\left(\left({y\over\lambda}\right)^{1/\theta}\right)
		=
		\mathcal{L}^{-1}\left\{	{\lambda\theta\over \mathbbm{A}(s^{1/\theta}) (s^{1/\theta}+1)^2}\,
		f_Z\left({1\over s^{1/\theta}+1}\right)\right\}(y).
	\end{align*}
	Setting $x=({y/\lambda})^{1/\theta}$ and making simple algebraic manipulations, the proof follows.

\subsection*{Proof of Theorem 2}
\noindent

By combining  the decomposition in \eqref{decomp-pdf-Y} with \eqref{decomp-pdf-Y-0-7}, and using the
identity \eqref{fZ-density}, we obtain	
\begin{align*}
{1\over \mathbbm{A}(s) (s+1)^2}\,
f_Z\left({1\over s+1}\right)
&=
\int_0^\infty
x \mathbbm{B}(x) f_X(x)  (p+qsx)
\exp(-\lambda s x)
{\rm d}x
\\[0,2cm]
&=
p
\mathcal{L}\left\{
x \mathbbm{B}(x) f_X(x)
\right\}(\lambda s)
+
qs
\mathcal{L}\left\{
x^2 \mathbbm{B}(x) f_X(x)
\right\}(\lambda s).
%\\[0,2cm]
%&=
%{p\over\lambda}\,
%\mathcal{L}\left\{
%\dfrac{x}{\lambda} \, B\left(\dfrac{x}{\lambda}\right) f_X\left(\dfrac{x}{\lambda}\right)
%\right\}(s)
%+
%{q\over\lambda}\,
%s
%\mathcal{L}\left\{
%\left(\dfrac{x}{\lambda}\right)^2 B\left(\dfrac{x}{\lambda}\right) f_X\left(\dfrac{x}{\lambda}\right)
%\right\}(s),
\end{align*}
%where in the last equality we used the scale change property of $\mathcal{L}$. That is,
%\begin{align*}
%{\lambda\over A(s) (s+1)^2}\,
%f_Z\left({1\over s+1}\right)
%=
%{p}
%\mathcal{L}\left\{
%\dfrac{x}{\lambda} \, B\left(\dfrac{x}{\lambda}\right) f_X\left(\dfrac{x}{\lambda}\right)
%\right\}(s)
%+
%{q}
%s
%\mathcal{L}\left\{
%\left(\dfrac{x}{\lambda}\right)^2 B\left(\dfrac{x}{\lambda}\right) f_X\left(\dfrac{x}{\lambda}\right)
%\right\}(s)
%\end{align*}
Now, by applying the inverse Laplace transform on both sides of the above equality and using the well-known property, that
\begin{align}\label{property-inv-Laplace}
\mathcal{L}^{-1}\{s\mathcal{L}\{g(x)\}(s)\}(x)= g'(x),
\end{align}
we obtain
\begin{multline*}
\mathcal{L}^{-1}\left\{
{\lambda^2\over \mathbbm{A}\left(\dfrac{s}{\lambda}\right) (s+\lambda)^2}\,
f_Z\left({\lambda\over s+\lambda}\right)
\right\}(x)
=
{p}
x \mathbbm{B}(x) f_X(x)
+
{q\over\lambda}\,
[x^2 \mathbbm{B}(x) f_X(x)]'
\\[0,2cm]
=
\left[
\left(p+{2q\over\lambda}\right)
x \mathbbm{B}(x)
+
{q\over\lambda}\,x^2 \mathbbm{B}'(x)
\right]
f_X(x)
+
{q\over\lambda}\,
x^2 \mathbbm{B}(x) f_X'(x);
\end{multline*}
equivalently,
\begin{align}\label{eq-1o-order}
	f_X'(x)+P(x)f_X(x)=Q(x),
\end{align}
where
\begin{align*}
P(x)
=	
\displaystyle
	\left({\lambda p\over q}+2\right) {1\over x}
	+
	{\mathbbm{B}'(x)\over \mathbbm{B}(x)},
\quad
Q(x)=	
{\lambda^3\over qx^2 \mathbbm{B}(x)}\,
\mathcal{L}^{-1}\left\{
{1\over \mathbbm{A}\left(\dfrac{s}{\lambda}\right) (s+\lambda)^2}\,
f_Z\left({\lambda\over s+\lambda}\right)
\right\}(x).
\end{align*}
Observe that Eq. \eqref{eq-1o-order} is a first-order non-homogeneous linear differential equation of type $y'+p(t)y=q(t)$, $t>0$, whose solution (by method of integrating factor), well-known in the field of differential equations, can be given as
$y(t)=\exp(-\int p(t){\rm d}t) [\int_{-\infty}^t \exp(\int p(s){\rm d}s) q(s){\rm d}s+{\rm const}]$, with const$=[\exp(\int p(t){\rm d}t) y(t)]\vert_{t=0^{+}}$. Hence, for $x>0$, we have
\begin{align*}
	f_X(x)
	&=
	\exp\left(-\int P(x){\rm d}x\right)
	\left[
	\int_0^x \exp\left(\int P(\xi){\rm d}\xi\right) Q(\xi){\rm d}\xi
	+
	{\rm const}
	\right]
	\\[0,2cm]
	&=
	{1\over x^{(\lambda p/ q)+2} \mathbbm{B}(x)}
	\left[
	\int_0^x
	\xi^{(\lambda p/ q)+2} \mathbbm{B}(\xi) Q(\xi){\rm d}\xi
	+
	{\rm const}
	\right]
		\\[0,2cm]
	&=
		{1\over x^{(\lambda p/ q)+2} \mathbbm{B}(x)}
	\left[
	{\lambda^3\over q}
	\int_0^x
	\xi^{\lambda p/ q}
	\mathcal{L}^{-1}\left\{
	{1\over \mathbbm{A}\left(\dfrac{s}{\lambda}\right) (s+\lambda)^2}\,
	f_Z\left({\lambda\over s+\lambda}\right)
	\right\}(\xi)
	{\rm d}\xi
	+
	{\rm const}
	\right],
\end{align*}
where const$=[x^{(\lambda p/ q)+2} \mathbbm{B}(x) f_X(x)]\vert_{x=0^+}$.
%
%
%\begin{align*}
%	\exp\left(\int P(x){\rm d}x\right)
%	=
%	\exp\left(	\left({\lambda p\over q}+2\right) \log(x)
%	+
%	\log[B(x)] \right)
%	=
%	x^{(\lambda p/ q)+2} B(x)
%\end{align*}
From the above identity, the proof follows.

\subsection*{Proof of Theorem 3}
\noindent

\begin{proof}
From \eqref{decomp-pdf-Y-1} and \eqref{decomp-pdf-Y}, we have
\begin{align}\label{decomp-pdf-Y-2}
f_Y(sx)
=
\mathbbm{A}(s) \mathbbm{B}(x) \mathbbm{C}(sx)
=
{1\over (\theta s)^p}\, \mathbbm{A}(s) \mathbbm{B}(x) \left({1\over  \theta s}+x\right)^{-p}, \quad x, s>0.
\end{align}
Now, by combining the identities in \eqref{fZ-density} and \eqref{decomp-pdf-Y-2}, we obtain
\begin{align*}
{ (\theta s)^p\over \mathbbm{A}(s) (s+1)^2}\,
f_Z\left({1\over s+1}\right)
&=
\int_0^\infty
x \mathbbm{B}(x) f_X(x)
\left({1\over  \theta s}+x\right)^{-p}
{\rm d}x
\\[0,2cm]
&=
\mathcal{G}_p\{x \mathbbm{B}(x) f_X(x) \}\left({1\over \theta s}\right).
\end{align*}
Applying the inverse generalized Stieltjes transform on both sides of the above equality, the proof of the required identity readily follows.
%
%
%\begin{align*}%\label{id-0}
%f_Z(z)
%=
%{1\over (\theta s)^p}\, A(s)  (s+1)^2
%\int_0^\infty x B(x) f_X(x)  \left({1\over  \theta s}+x\right)^{-p} {\rm d}x,
%\ \ {\mbox{where}}\ \ s={1\over z}-1.
%\end{align*}
%Equivalently,
%\begin{align*}
%{(\theta s)^p\over A(s) (s+1)^2}\,
%f_Z\left({1\over s+1}\right)
%&=
%\int_0^\infty x B(x) f_X(x)   \left({1\over  \theta s}+x\right)^{-p} {\rm d}x
%\\[0,2cm]
%&=
%\mathcal{G}_p\{x B(x) f_X(x) \}\left({1\over \theta s}\right),
%\end{align*}
%where $\mathcal{G}_p$ is the generalized Stieltjes transform \citep{Schwarz05}. Using the formula (23) of \cite{Schwarz05}, the above expression is
%\begin{align*}
%	={1\over \Gamma(p)} \,
%	\mathcal{L}\big\{t^{p-1} \mathcal{L}\{x B(x) f_X(x)\}(t)\big\}\left({1\over \theta s}\right).
%\end{align*}
%Therefore,
%\begin{align*}
%{\Gamma(p)(\theta s)^p\over A(s) (s+1)^2}\,
%f_Z\left({1\over s+1}\right)
%=
%\mathcal{L}\big\{t^{p-1} \mathcal{L}\{x B(x) f_X(x)\}(t)\big\}\left({1\over \theta s}\right).
%\end{align*}
%Taking the inverse Laplace transform on both sides of the above equality, we obtain
%\begin{align*}
%x B(x) f_X(x)
%=
%\Gamma(p)
%	\mathcal{L}^{-1}\left\{
%t^{1-p}
%\mathcal{L}^{-1}\left\{
%{1\over s^{p} A\left(\dfrac{1}{\theta s}\right) }
%\left({\theta s\over 1+\theta s}\right)^2
%f_Z\left({\theta s\over 1+ \theta s}\right)
%\right\}(t)
%\right\}(x).
%\end{align*}
%This completes the proof.
\end{proof}

\end{appendices}

\end{document}